\documentclass{amsart}

\usepackage{enumerate}

\DeclareMathOperator{\tw}{tw}

\newtheorem{lemma}{Lemma}
\newtheorem{theorem}{Theorem}
\newtheorem{claim}{Claim}
\begin{document}
\title{A simple proof of the tree-width duality theorem}

\author{Fr\'ed\'eric Mazoit}%
\email{Frederic.Mazoit@labri.fr}
\thanks{This research was supported by the french ANR project DORSO.}
\address{LaBRI, Universit\'e de Bordeaux\\
  351 cours de la libération, F-33405 Talence CEDEX, France.}

\begin{abstract}
  We give a simple proof of the ``tree-width duality theorem'' of
  Seymour and Thomas that the tree-width of a finite graph is exactly
  one less than the largest order of its brambles.
\end{abstract}
\maketitle

\section{Introduction}
A \emph{tree-decomposition} $\mathcal{T}=(T, l)$ of a graph $G=(V, E)$
is tree whose nodes are labelled in such a way that
\begin{enumerate}[i.]
\item $V=\bigcup_{t\in V(T)}l(t)$;
\item every $e\in E$ is contained in at least one $l(t)$;
\item for every vertex $v\in V$, the nodes of $T$ whose bags contain
  $v$ induce a connected subtree of $T$.
\end{enumerate}
The label of a node is its \emph{bag}.  The \emph{width} of
$\mathcal{T}$ is $\max\{|l(t)|\;;\; t\in V(T)\}-1$, and the
\emph{tree-width} $\tw(G)$ of $G$ is the least width of any of its
tree-decomposition.

Two subsets $X$ and $Y$ of $V$ \emph{touch} if they meet or if there
exists an edge linking them.  A set $\mathcal{B}$ of mutually touching
connected vertex sets in $G$ is a \emph{bramble}.  A \emph{cover} of
$\mathcal{B}$ is a set of vertices which meets all its elements, and
the \emph{order} of $\mathcal{B}$ is the least size of one of its
covers.

In this note, we give a new proof of the following theorem of Seymour
and Thomas which Reed~\cite{Re97a} calls the ``tree-width duality
theorem''.
\begin{theorem}[\cite{SeTh93a}]
  Let $k\geq 0$ be an integer. A graph has tree-width $\geq k$ if and
  only if it contains a bramble of order $>k$.
\end{theorem}
Although our proof is quite short, our goal is not to give a shorter
proof.  The proof in~\cite{Di05b} is already short enough.  Instead,
we claim that our proof is much simpler than previous ones.  Indeed,
the proofs in~\cite{SeTh93a,Di05b} rely on a reverse induction on the
size of a bramble which is not very enlightening.  A new conceptually
much simpler proof appeared in~\cite{LyMaTh10a} but this proof is a
much more general result on sets of partitions which through a
translation process unifies all known duality theorem of this kind
such as the branch-width/tangle or the path-width blockade Theorems.
We turn this more general proof back into a specific proof for
tree-width which we believe is interesting both as an introduction to
the framework of~\cite{AmMaNiTh09a, LyMaTh10a}, and to a reader which
does not want to dwell into this framework but still want to have a
better understanding of the tree-width duality Theorem.

\section{The proof}
So let $G=(V, E)$ be a graph and let $k$ be a fixed integer.  A bag of
a tree-decomposition of $G$ is \emph{small} if it has size $\leq k$
and is \emph{big} otherwise.  A \emph{partial $(<k)$-decomposition} is
a tree-decomposition $\mathcal{T}$ with no big internal bag and with
at least one small bag.  Obviously, if all its bags are small, then
$\mathcal{T}$ is a tree-decomposition of width $<k$.  If not, it
contains a big leaf bag and the neighbouring bag $l(u)$ of any such
big leaf bag $l(t)$ is small.  The nonempty set $l(t)-l(u)$ is a
\emph{$k$-flap} of $\mathcal{T}$.

Now suppose that $X$ and $Y$ are respectively $k$-flaps of some
partial $(<k)$-decompositions $(T_X, l_X)$ and $(T_Y, l_Y)$, and that
$S=N(X)\subseteq N(Y)$.  Then by identifying the leaves of the two
decompositions which respectively contains $X$ and $Y$ and relabelling
this node $S$, then we obtain a new ``better'' partial
$(<k)$-decomposition.

This gluing process is quite powerful.  Indeed let $S\subseteq V$ have
size $\leq k$ and let $C_1$,~\dots, $C_p$ be the components of $G-S$.
The star whose centre $u$ is labelled $l(u)=S$ and whose $p$ leaves
$v_1$,~\dots,~$v_p$ are labelled by $l(v_i)=C_i\cup N(C_i)$ is a
partial $(<k)$-decomposition which we call the \emph{star
  decomposition from $S$}.  It can be shown that if $\tw(G)<k$, then
an optimal tree-decomposition can always be obtained by repeatedly
applying this gluing process from star decompositions from sets of
size $\leq k$.  But this process is not powerful enough for our
purpose.  We need the following lemma.
\begin{lemma}\label{lem}
  Let $X$ and $Y$ be respectively $k$-flaps of some partial
  $(<k)$-decompositions $(T_X, l_X)$ and $(T_Y, l_Y)$ of some graph
  $G=(V, E)$.  If $X$ and $Y$ do not touch, then there exists a
  partial $(<k)$-decomposition $(T, l)$ whose $k$-flaps are subsets of
  $k$-flaps of $(T_X, l_X)$ and $(T_Y, l_Y)$ other than $X$ and $Y$.
\end{lemma}
\begin{proof}
  Since, $X$ and $Y$ no not touch, there exists $S\subseteq V$ such
  that no component of $G-S$ meet both $X$ and $Y$ (for example
  $N(X)$). Choose such an $S$ with $|S|$ minimal.  Note that $|S|\leq
  |N(X)|\leq k$.  Let $A$ contain $S$ and all the components of $G-S$
  which meet $X$, and let $B=(V-A)\cup S$.

  \begin{claim}
    There exists a partial $(<k)$-decomposition of $G[B]$ with $S$ as
    a leaf and whose $k$-flaps are subsets of the $k$-flaps of $(T_X,
    l_X)$ other than $X$.
  \end{claim}
  Let $x$ be the leaf of $T_X$ whose bag contains $X$.  Since $|S|$ is
  minimum, there exists $|S|$ vertex disjoint paths $P_s$ from $X$ to
  $S$ ($s\in S$).  Note that $P_s$ only meets $B$ in $s$.  For each
  $s\in S$, pick a node $t_s$ in $T_X$ with $s\in l_X(t_s)$, and let
  $l'_X(t)=(l_X(t)\cap B)\cup \{s| t\in \text{path from }x\text{ to
  }t_s\}$ for all $t\in T$.  Then $(T_X, l'_X)$ is the
  tree-decomposition of $G[B]$.  Indeed, since we removed only
  vertices not in $B$, every vertex and every edge of $G[B]$ is
  contained in some bag $l'_X(t)$.  Moreover, for any $v\notin S$,
  $l'_X(t)$ contains $v$ if and only if $l_X(t)$ does.  And $l'_X(t)$
  contains $s\in S$ if $l_X(t)$ does or if $t$ is on the path from $x$
  to $t_s$.  In either cases, the vertices $t\in V(T_X)$ whose bag
  $l'_X(t)$ contain a given vertex induce a subtree of $T_X$.

  Now the size of a bag $l'_X(t)$ is at most $|l_X(t)|$.  Indeed,
  since $P_s$ is a connected subgraph of $G$, it induces a connected
  subtree of $T_X$, and this subtree contains the path from $x$ to
  $t_s$. So for every vertex $s\in l'_X(t)\setminus l_X(t)$, there
  exists at least one other vertex of $P_s$ which as been removed.
  The decomposition $(T_X, l'_X)$ is thus indeed a partial
  $(<k)$-decomposition of $G[B]$.  It remains to prove that the
  $k$-flaps of $(T_X, l'_X)$ are contained in the $k$-flaps of $(T_X,
  l_X)$ other than $X$.  But by construction, the only leaf whose bag
  received new vertices is $x$ and $l'_X(x)=S$ which is small.  This
  finishes the proof of the claim.

  Let $(T_Y, l'_Y)$ be obtains in the same way for $G[A]$.  By
  identifying the leaves $x$ and $y$ of $T_X$ and $T_Y$, we obtain a
  partial $(<k)$-decomposition which satisfies the conditions of the
  lemma.
\end{proof}

We are now ready to prove the tree-width duality Theorem.
\begin{proof}
  For the backward implication, let $\mathcal{B}$ be a bramble of
  order $>k$ in a graph $G$.  We show that every tree-decomposition
  $(T, l)$ of $G$ has a part that covers $\mathcal{B}$, and thus
  $\mathcal{T}$ has width $\geq k$.

  We start by orienting the edges $t_1t_2$ of $T$.  Let $T_i$ be the
  component of $T\setminus t_1t_2$ which contains $t_i$ and let
  $V_i=\cup_{t\in V(T_i)}l(t)$.  If $X:=l(t_1)\cap l(t_2)$ covers
  $\mathcal{B}$, we are done. If not, then because they are connected,
  each $B\in\mathcal{B}$ disjoint from $X$ in contained is some
  $B\subseteq V_i$.  This $i$ is the same for all such $B$, because
  they touch.  We now orient the edge $t_1t_2$ towards $t_i$.  If
  every edge of $T$ is oriented in this way and $t$ is the last vertex
  of a maximal directed path in $T$, then $l(t)$ covers $\mathcal{B}$.

  To prove the forward direction, we now assume that $G$ has
  tree-width $\geq k$, then any partial $(<k)$-decomposition contains
  a $k$-flap.  There thus exists a set $\mathcal{B}$ of $k$-flaps such
  that
  \begin{enumerate}[(i)]
  \item $\mathcal{B}$ contains a flap of every partial
    $(<k)$-decomposition;

  \item $\mathcal{B}$ is upward closed, that is if $C\in\mathcal{B}$
    and $D\supseteq C$ is a $k$-flap, then $D\in\mathcal{B}$.
  \end{enumerate}
  So far, the set of all $k$-flaps satisfies $(i)$ and $(ii)$.

  \begin{enumerate}
  \item[(iii)] Subject to $(i)$ and $(ii)$, $\mathcal{B}$ is
    inclusion-wise minimal.
  \end{enumerate}
  The set $\mathcal{B}$ may not be a bramble because it may contain
  non-connected elements but we claim that the set $\mathcal{B}'$
  which contains the connected elements of $\mathcal{B}$ is a bramble
  of order $\geq k$.  Obviously, its elements are connected.  To see
  that its order is $> k$, let $S\subseteq V$ have size $\leq k$.
  Then $\mathcal{B}'$ contains a $k$-flap of the star-decomposition
  from $S$, and $S$ is thus not a covering of $\mathcal{B}'$.

  We now prove that the elements of $\mathcal{B}$ pairwise touch,
  which finishes the proof that $\mathcal{B}'$ is a bramble.  Suppose
  not, then let $X$ and $Y\in\mathcal{B}$ witness this.  Obviously, no
  subsets of $X$ and $Y$ can touch so let us suppose that they are
  inclusion-wise minimal in $\mathcal{B}$.  The set $X$ being minimal,
  $\mathcal{B}\setminus\{X\}$ is still upward closed and is a strict
  subset of $\mathcal{B}$.  There thus exists at least one partial
  $(<k)$-decomposition $(T_X, l_X)$ whose only flap in $\mathcal{B}$
  is $X$.  Likewise, let $(T_Y, l_Y)$ have only $Y$ as a flap in
  $\mathcal{B}$.  Let $(T, l)$ be the partial $(<k)$-decomposition
  satisfying the conditions of Lemma~\ref{lem}.  Since $\mathcal{B}$
  is upward closed and contains no $k$-flap of $(T_X, l_X)$ and $(T_Y,
  l_Y)$ other than $X$ and $Y$, it contains no $k$-flap of $(T, l)$, a
  contradiction.
\end{proof}

\end{document}